\documentclass{article}

\usepackage[square,sort&compress,numbers]{natbib}

\usepackage[australian]{babel}
\usepackage{fullpage}

\usepackage{hyperref}

\usepackage{amssymb,amsmath,amsfonts}
\usepackage{amsthm}
\usepackage[all]{xy}
\usepackage[capitalise]{cleveref}

\allowdisplaybreaks
\numberwithin{equation}{section}

\usepackage{xcolor}
\definecolor{MyBlue}{cmyk}{1,0.13,0,0.63}
\definecolor{MyGreen}{cmyk}{0.91,0,0.88,0.52}
\newcommand{\mylinkcolor}{MyBlue}
\newcommand{\mycitecolor}{MyGreen}
\newcommand{\myurlcolor}{webbrown}

\makeatletter
\def\@endtheorem{\endtrivlist}
\makeatother
\theoremstyle{plain}
\newtheorem{thm}{Theorem}[section]
\newtheorem{lem}[thm]{Lemma}
\newtheorem{prop}[thm]{Proposition}
\newtheorem{coro}[thm]{Corollary}

\theoremstyle{definition}
\newtheorem{defn}[thm]{Definition}
\newtheorem{remark}[thm]{Remark}

\newtheorem{example}[thm]{Example}

\newtheoremstyle{note}
{3pt}
{3pt}
{\bfseries}
{\parindent}
{\bfseries\itshape}
{:}
{.5em}
{}
\theoremstyle{note}
\newtheorem*{Question}{Question}
\newtheorem*{Problem}{Problem}
\newtheorem*{Note}{Note}

\renewcommand{\eqref}[1]{\labelcref{#1}}
\crefname{thm}{Theorem}{Theorems}
\crefname{lem}{Lemma}{Lemmas}
\crefname{prop}{Proposition}{Propositions}
\crefname{coro}{Corollary}{Corollaries}
\crefname{defn}{Definition}{Definitions}
\crefname{example}{Example}{Examples}
\crefname{remark}{Remark}{Remarks}

\hypersetup{%
  bookmarksnumbered=true,bookmarksopen=false,%
  plainpages=false,
  linktocpage=true,%
  colorlinks=true,breaklinks=true,%
  linkcolor=\mylinkcolor,citecolor=\mycitecolor,urlcolor=\myurlcolor,%
  pdfpagelayout=OneColumn,%
  pageanchor=true,%
}

\usepackage{enumitem}
\setlist{topsep=4pt plus 2pt minus 2pt,partopsep=0pt,itemsep=2pt plus 2pt minus 2pt,parsep=0.5\parskip}
\setenumerate[1]{label=\arabic*)}

\let\OLDthebibliography\thebibliography
\renewcommand\thebibliography[1]{
  \addcontentsline{toc}{section}{\refname}
  \OLDthebibliography{CPRS06b}
  \setlength{\parskip}{0pt}
  \setlength{\itemsep}{0pt plus 0.3ex}
}

\newcommand{\N}{\mathbb{N}}
\newcommand{\R}{\mathbb{R}}
\newcommand{\C}{\mathbb{C}}
\newcommand{\Z}{\mathbb{Z}}
\newcommand{\A}{\mathcal{A}}
\newcommand{\mH}{\mathcal{H}}
\newcommand{\D}{\mathcal{D}}
\newcommand{\E}{\mathcal{E}}
\newcommand{\B}{\mathcal{B}}

\DeclareMathOperator{\Dom}{Dom}
\DeclareMathOperator{\End}{End}

\DeclareMathOperator{\Lip}{Lip}
\DeclareMathOperator{\Span}{span}
\DeclareMathOperator{\Ran}{Ran}

\DeclareMathOperator*{\esssup}{ess\,sup}

\renewcommand{\bar}[1]{\overline{#1}}

\newcommand{\CCliff}{{\mathbb{C}\mathrm{l}}}

\newcommand{\KK}{\textnormal{KK}}

\newcommand{\til}[1]{\tilde{#1}}

\newcommand{\hotimes}{\mathbin{\hat\otimes}}
\newcommand{\hot}{\hotimes}
\newcommand{\hodot}{\mathbin{\hat\odot}}
\newcommand{\la}{\langle}
\newcommand{\ra}{\rangle}

\newcommand{\bV}{V}

\newcommand{\mattwo}[4]{
  \left(\!\!\!\begin{array}{c@{~}c}#1&#2\\ #3&#4\\\end{array}\!\!\!\right)
}

\title{Locally bounded perturbations and (odd) unbounded KK-theory}
\author{
Koen van den Dungen%
\footnote{\emph{Present address:} Mathematisches Institut der Universit\"at Bonn, Endenicher Allee 60, D-53115 Bonn, \texttt{kdungen@uni-bonn.de}}\\[4mm]
{\normalsize SISSA (Scuola Internazionale Superiore di Studi Avanzati)}\\ 
{\normalsize Via Bonomea, 265, 34136 Trieste, Italy}
}

\begin{document}

\maketitle

\begin{abstract}
\noindent
A regular symmetric operator on a Hilbert module is self-adjoint whenever there exists a suitable approximate identity. We say an operator is `locally bounded' if the composition of the operator with each element in the approximate identity is bounded. We prove that the perturbation of a regular self-adjoint operator by a locally bounded symmetric operator is again regular and self-adjoint. We use this result to show that the Kasparov class represented by an unbounded Kasparov module is stable under locally bounded perturbations. As an application, we show that we obtain a converse to the `doubling up' procedure of odd unbounded Kasparov modules. Finally, we discuss perturbations of unbounded Kasparov modules by unbounded multipliers. In particular, we explicitly construct an unbounded multiplier such that (after doubling up the module) the perturbed operator has compact resolvent. 

\vspace{\baselineskip}

\end{abstract}

\section{Introduction}

It is a classical result by Chernoff \cite{Che73} that any symmetric first-order differential operator $\D$ with bounded propagation speed on a complete Riemannian manifold $X$ must be essentially self-adjoint. 
One way to prove this statement (following the argument in \cite[\S1]{GL83}) is by using the fact that there exist compactly supported functions $\phi_k\in C_c^\infty(X,\R)$ (for $k\in\N$), converging pointwise to $1$, such that $[\D,\phi_k] \to 0$, along with the fact that $\phi_k\cdot\Dom\D^* \subset \Dom\D$. 

More abstractly, if $\D$ is a symmetric operator on a Hilbert space $\mH$, we say an approximate identity $\{\phi_k\}_{k\in\N}\subset\B(\mH)$ is \emph{adequate} \cite{MR16} if $\phi_k\cdot\Dom\D^* \subset \Dom\D$ and the commutators $[\D,\phi_k]$ are well-defined and uniformly bounded. 
We view the existence of an adequate approximate identity for $\D$ as a generalisation of the classical assumptions that $\D$ is first-order and has bounded propagation speed, and that the underlying Riemannian manifold is complete. 
We prove in \cref{sec:approx_id} that (as in the classical case) the existence of such an approximate identity implies that $\D$ is essentially self-adjoint. More generally, using the local-global principle \cite{Pie06,KL12}, we can extend this result to Hilbert modules: given a regular symmetric operator $\D$ on a Hilbert $B$-module $E$ and an adequate approximate identity $\{\phi_k\}_{k\in\N}\subset\End_B(E)$, it follows that $\D$ is self-adjoint. 

Let $\D$ again be a symmetric first-order differential operator with bounded propagation speed on a complete Riemannian manifold, and let $T$ be any symmetric zeroth-order operator. 
Since the propagation speed of a differential operator depends only on the principal symbol, we know that $\D+T$ is again essentially self-adjoint, no matter how unbounded the perturbation $T$ might be (for the case of smooth perturbations (`potentials') $T$, this situation was already dealt with by Chernoff \cite{Che73}). 
In \cref{sec:loc_bdd} we provide an abstract analogue of this statement. Note that the restriction of $T$ to a compact subset of the manifold is bounded. 

Abstractly, for a densely defined operator $M$ on a Hilbert module $E$, and for an adequate approximate identity $\{\phi_k\}_{k\in\N}$, we say that $M$ is \emph{locally bounded} if $M\phi_k$ is well-defined and bounded for each $k\in\N$. One of the main results of this article is then that, given the existence of a suitable approximate identity $\{\phi_k\}_{k\in\N}$, the perturbation of a regular self-adjoint operator $\D$ by a locally bounded symmetric operator $M$ is again regular self-adjoint (\cref{thm:reg_sa_sum_loc_bdd}). 
Though local boundedness is of course a strong assumption on the perturbation $M$, the main novelty of this result is that (unlike e.g.\ the well-known Kato-Rellich theorem or W\"ust's theorem) we do not assume any relative bound on the perturbation.

In \cref{sec:Kasp_mod_stab} we apply our result to the framework of noncommutative geometry and unbounded $\KK$-theory. 
We prove that (again given the existence of a suitable approximate identity) a spectral triple, or more generally an unbounded Kasparov module, is stable under locally bounded perturbations. This provides an unbounded analogue of the fact that the class of a bounded Kasparov module is stable under locally compact perturbations (see \cite[Proposition 17.2.5]{Blackadar98}). 

As an application, we will have a look at the odd version(s) of unbounded $\KK$-theory in \cref{sec:odd_KK}. For trivially graded $C^*$-algebras $A$ and $B$, we consider two types of unbounded representatives for a class in the odd $\KK$-theory $\KK^1(A,B) = \KK(A\otimes\CCliff_1,B)$:
\begin{enumerate}
\item an \emph{odd} unbounded Kasparov $A$-$B$-module $(\A,E_B,\D)$ (where the Hilbert module $E$ is trivially graded);
\item an (even) unbounded Kasparov $A\otimes\CCliff_1$-$B$-module $(\A\otimes\CCliff_1,\til E_B,\til\D)$.
\end{enumerate}
Any \emph{odd} unbounded Kasparov $A$-$B$-module $(\A,{}_\pi E_B,\D)$ can straightforwardly be `doubled up' to an (even) unbounded Kasparov $A\otimes\CCliff_1$-$B$-module $(\A\otimes\CCliff_1,{}_{\til\pi}\til E_B,\til\D)$ for which $\til\D$ anti-commutes with the generator of the Clifford algebra $\CCliff_1$ (see \cref{sec:odd_KK}). 
Conversely, however, given an arbitrary (even) unbounded Kasparov $A\otimes\CCliff_1$-$B$-module $(\A\otimes\CCliff_1,{}_{\til\pi}\til E_B,\til\D)$, the operator $\til\D$ does not need to anti-commute with the Clifford generator (we only know that $\til\D$ has bounded graded commutators with the algebra). 
Thus, if we wish to reduce the even module $(\A\otimes\CCliff_1,{}_{\til\pi}\til E_B,\til\D)$ to an \emph{odd} unbounded Kasparov $A$-$B$-module, we need to show that we can replace $\til\D$ by $\til\D' := \frac12(\til\D-e\til\D e)$ (where $e$ denotes the generator of $\CCliff_1$) without changing the underlying class in $\KK$-theory.  
By observing that $\til M := \frac12(\til\D+e\til\D e)$ is locally bounded, it then follows from the stability of unbounded Kasparov modules under locally bounded perturbations that (the closure of) $\til\D' = \til\D - \til M$ indeed represents the same class as $\til\D$. 

Finally, in \cref{sec:unbdd_mult} we consider the natural example of a locally bounded perturbation of an unbounded Kasparov module $(\A,E_B,\D)$ arising from a (symmetric) unbounded multiplier on the (typically non-unital) algebra $\A$. 
If $(\A,E_B,\D)$ is an unbounded Kasparov $A$-$B$-module for a non-unital $C^*$-algebra $A$, then in general the resolvent of $\D$ is only \emph{locally} compact. In practice, it can be much easier to deal with operators whose resolvent is in fact compact. 
In \cref{sec:cpt_res} we give sufficient conditions which ensure that we can find a locally bounded perturbation such that the perturbed operator has compact resolvent. 
In fact, we will construct this locally bounded perturbation explicitly as an unbounded multiplier built from a given adequate approximate identity. More precisely, we show that for a given odd module $(\A,E_B,\D)$ we can explicitly construct an unbounded multiplier on $A$ such that the perturbation of the `doubled up' module $(\A\otimes\CCliff_1,\til E_B,\til\D)$ by this unbounded multiplier has compact resolvent. We provide a similar statement in the even case, where the `doubling up' is based on the isomorphism $\KK(A,B) \simeq \KK^2(A,B) = \KK(A\hot\CCliff_2,B)$.

\subsection*{Acknowledgements}

Many thanks to Bram Mesland and Adam Rennie for interesting discussions and helpful suggestions. Thanks also to Magnus Goffeng for suggesting the question addressed in \cref{sec:cpt_res}. Finally, thanks to the referee for helpful comments and suggestions.

\section{Approximate identities}
\label{sec:approx_id}

Let $B$ be a $\Z_2$-graded $C^*$-algebra. Recall that a $\Z_2$-graded Hilbert $B$-module $E$ is a vector space equipped with a $\Z_2$-graded right action $E\times B\to E$ and with a $B$-valued inner product $\la\cdot|\cdot\ra\colon E\times E\to B$, such that $E$ is complete in the corresponding norm. 
The endomorphisms $\End_B(E)$ are the adjointable linear operators $E\to E$, and the set $\End^0_B(E)$ of compact endomorphisms is given by the closure of the finite rank operators. 
For an operator $T$ on $E$ we write $\deg T=0$ if $T$ is even and $\deg T=1$ if $T$ is odd. 
For a detailed introduction to Hilbert modules and $\Z_2$-gradings, we refer to \cite{Blackadar98,Lance95}. 

A densely defined operator $S$ on $E$ is called \emph{semi-regular} if the adjoint $S^*$ is densely defined. A semi-regular operator $S$ is closable, and we denote its closure by $\bar S$. A semi-regular operator $S$ is called \emph{regular} if $S$ is closed and $1+S^*S$ has dense range. 

If $B=\C$, then a Hilbert $\C$-module is just a Hilbert space $\mH$, and we write $\B(\mH) = \End_\C(\mH)$. In this case, any closed operator on $\mH$ is regular. 

\begin{defn}
A \emph{sequential approximate identity} on a Hilbert $B$-module $E$ is a sequence of self-adjoint operators $\phi_k\in \End_B(E)$ (for $k\in\N$) such that 
$\phi_k$ converges strongly to the identity on $E$. 
\end{defn}
Since $\phi_k\psi\to\psi$ for each $\psi\in E$, we have in particular that $\sup_{k\in\N}\|\phi_k\psi\|<\infty$ for each $\psi\in E$. The uniform boundedness principle then implies that $\sup_{k\in\N}\|\phi_k\|<\infty$. 

The assumption of self-adjointness is only imposed for convenience; in general, one could consider an arbitrary sequence $\{\phi_k\}_{k\in\N}$ which converges \emph{strictly} to the identity (i.e., $\phi_k\psi\to\psi$ and $\phi_k^*\psi\to\psi$ for all $\psi\in E$), and then $\frac12(\phi_k+\phi_k^*)$ gives a self-adjoint approximate identity. 

For any endomorphism $T\in\End_B(E)$ we have that $\phi_kT$ converges strongly to $T$. If $T$ is compact, we in fact have that $\phi_kT$ converges to $T$ in norm (which can be shown by first checking the norm convergence for finite rank operators). Hence, if $\phi_k\in\End_B^0(E)$ for each $k\in\N$, then $\{\phi_k\}_{k\in\N}$ is also a sequential approximate unit in the algebra $\End_B^0(E)$. 

\begin{defn}
Let $\D$ be an unbounded symmetric operator on a Hilbert $B$-module $E$. An \emph{adequate approximate identity for $\D$} is a sequential approximate identity $\{\phi_k\}_{k\in\N}$ on $E$ such that $\phi_k\cdot\Dom\D^*\subset\Dom\bar\D$, $[\bar\D,\phi_k]$ is bounded on $\Dom\D$ for all $k$, and $\sup_{k\in\N}\|\bar{[\bar\D,\phi_k]}\| < \infty$. 
\end{defn}

\begin{remark}
The term \emph{adequate approximate identity} is borrowed from \cite[\S2]{MR16}. 
This notion is weaker than the notion of a bounded approximate unit for the Lipschitz algebra $\Lip(\D)$ (see \cite{MR16} for details). 
\end{remark}

The following lemma shows that the commutators $[\bar\D,\phi_k]$ on $\Dom\bar\D$ and $[\D^*,\phi_k]$ on $\Dom\D^*$ are also bounded, and that their closures equal the closure of $[\bar\D,\phi_k]$ defined on $\Dom\D$. 

\begin{lem}
\label{lem:comm_D*}
Let $\D$ be a symmetric operator on a Hilbert $B$-module $E$, and let $\phi\in\End_B(E)$ be a self-adjoint operator such that $\phi\cdot\Dom\D\subset\Dom\bar\D$ and $[\bar\D,\phi]$ is bounded on $\Dom\D$. 
Then 
\begin{enumerate}
\item $\phi\cdot\Dom\bar\D\subset\Dom\bar\D$, $[\bar\D,\phi]$ is bounded on $\Dom\bar\D$ and $\bar{[\bar\D,\phi]} = \bar{[\bar\D,\phi]|_{\Dom\D}}$; and
\item $\phi\cdot\Dom\D^*\subset\Dom\D^*$ and $\bar{[\D^*,\phi]} = \bar{[\bar\D,\phi]}$. 
\end{enumerate}
\end{lem}
\begin{proof}
The first statement is proven in \cite[Proposition 2.1]{FMR14} for Hilbert spaces, but the proof also works for (semi-regular) operators on Hilbert modules. 
The proof of the second statement is similar and goes as follows. 
First, we observe that $\bar{[\bar\D,\phi]}$ is adjointable (indeed, its adjoint is densely defined on $\Dom\D$ and equal to $-[\bar\D,\phi]$, which is bounded). 
For $\psi\in\Dom\D$ and $\xi\in\Dom\D^*$ we have
$$
\la\phi\xi|\D\psi\ra = \la\xi|\phi\D\psi\ra = \la\xi|(\bar\D\phi-[\bar\D,\phi])\psi\ra = \la(\phi\D^*-[\bar\D,\phi]^*)\xi|\psi\ra ,
$$
which shows that $\phi\xi\in\Dom\D^*$, and hence that $[\D^*,\phi]$ is well-defined on $\Dom\D^*$. 
Restricted to $\Dom\bar\D$ we have $[\D^*,\phi]|_{\Dom\bar\D} = [\bar\D,\phi]$. If $[\D^*,\phi]$ is closable, this means that $\bar{[\D^*,\phi]} \supset \bar{[\bar\D,\phi]}$ and hence that $[\D^*,\phi]$ is bounded and $\bar{[\D^*,\phi]} = \bar{[\bar\D,\phi]}$. 
For $\psi\in\Dom\D$ and $\xi\in\Dom\D^*$ we have
$$
\la\psi|[\D^*,\phi]\xi\ra = \la\psi|(\D^*\phi-\phi\D^*)\xi\ra = \la(\phi\D-\bar\D\phi)\psi|\xi\ra = \la-[\bar\D,\phi]\psi|\xi\ra ,
$$
which means that $\Dom[\D^*,\phi]^* \supset \Dom\D$, which is dense in $E$. 
Thus $[\D^*,\phi]$ is indeed closable, and its closure equals $\bar{[\bar\D,\phi]}$.  
\end{proof}

If $\D$ is a regular self-adjoint operator, then $\{(1+\frac1k\D^2)^{-\frac12}\}_{k\in\N}$ is an adequate approximate identity for $\D$. 
The main reason for introducing the notion of an adequate approximate identity for $\D$ is that a converse statement also holds: if there exists an adequate approximate identity for a regular symmetric operator $\D$, then $\D$ is essentially self-adjoint. We will first prove this in the special case where $\D$ is an operator on a Hilbert space, and for this purpose we recall the following lemma (a proof can be found in e.g.\ \cite[Lemma 1.8.1]{Higson-Roe00}). 

\begin{lem}
\label{lem:op_closure}
Let $T$ be a closable operator on a Hilbert space $\mH$. Then $\psi\in\Dom\bar T$ if and only if there exists a sequence $\{\psi_k\}_{k\in\N}$ in $\Dom T$ such that $\psi_k\to\psi$ while $\|T\psi_k\|$ remains bounded. 
\end{lem}

\begin{prop}
\label{prop:approx_id_sa}
Let $\D$ be a symmetric operator on a Hilbert space $\mH$, and suppose we have an adequate approximate identity for $\D$. Then $\D$ is essentially self-adjoint. 
\end{prop}
\begin{proof}
Let $\{\phi_k\}_{k\in\N}\in\B(\mH)$ be an adequate approximate identity for $\D$. 
From \cref{lem:comm_D*} we know that $\bar{[\D^*,\phi_k]} = \bar{[\bar\D,\phi_k]}$, and in particular this shows that $[\D^*,\phi_k]$ is uniformly bounded. 
For $\xi\in\Dom\D^*$ we have $\phi_k\xi\in\Dom\bar\D$, and we find that
$$
\bar\D\phi_k\xi = \D^*\phi_k\xi = \phi_k\D^*\xi + [\D^*,\phi_k]\xi 
$$
is a bounded sequence in $\mH$. Since $\phi_k\xi\to\xi$, we conclude from \cref{lem:op_closure} that $\xi\in\Dom\bar\D$ and hence that $\bar\D$ is self-adjoint. 
\end{proof}

The proof of \cref{lem:op_closure} uses that every bounded sequence in $\mH$ has a weakly convergent subsequence, which relies on the fact that a Hilbert space is equal to its own dual. Since a Hilbert $B$-module is in general \emph{not} equal to its own dual, the proof does not generalise to Hilbert modules. Instead, we will invoke the local-global principle to prove the analogue of \cref{prop:approx_id_sa} for Hilbert modules. 

Let us briefly recall the local-global principle from \cite{Pie06} (see also \cite{KL12}). Let $E_B$ be a right Hilbert $B$-module, and let $\pi$ be a representation of $B$ on a Hilbert space $\mH^\pi$. We then get an induced representation $\pi_E$ of $\End_B(E)$ on the interior tensor product $E\hot_B\mH^\pi$. This Hilbert space $E\hot_B\mH^\pi$ is called the \emph{localisation} of $E$ with respect to the representation $\pi$. 

Now let $T$ be a semi-regular operator on $E_B$. We define the unbounded operator $T_0^\pi$ on $E\hot_B\mH^\pi$ as $T_0^\pi(e\hot h) := (Te)\hot h$ with domain $\Dom T\hat\odot_B\mH^\pi$ (where $\hat\odot$ denotes the algebraic tensor product). Then $T_0^\pi$ is densely defined and closable, and its closure $T^\pi$ is called the \emph{localisation} of $T$ with respect to $\pi$. We have the inclusion $(T^*)^\pi \subset (T^\pi)^*$. In particular, if $T$ is symmetric, then so is $T^\pi$. 

\begin{thm}[{Local-global principle \cite[Th\'eor\`eme 1.18]{Pie06}}]
\label{thm:loc-glob}
For a closed, densely defined and symmetric operator $T$ on a Hilbert module $E_B$, the following statements are equivalent:
\begin{enumerate}
\item the operator $T$ is self-adjoint and regular;
\item for every irreducible representation $(\pi,\mH^\pi)$ of $B$ the localisation $T^\pi$ is self-adjoint. 
\end{enumerate}
\end{thm}

\begin{lem}
\label{lem:loc_sum_comp}
Let $S$ and $T$ be semi-regular operators on a Hilbert $B$-module $E$, and let $\pi\colon B\to\B(\mH^\pi)$ be a $*$-representation of $B$. 
\begin{enumerate}
\item If $S+T$ is semi-regular and $S^\pi+T^\pi$ is closable, then $(S+T)^\pi \subset \bar{S^\pi+T^\pi}$.
\item If $ST$ is semi-regular and $S^\pi T^\pi$ is closable, then $(ST)^\pi \subset \bar{S^\pi T^\pi}$. 
\end{enumerate}
\end{lem}
\begin{proof}
\begin{enumerate}
\item By definition we have $\Dom(S+T) := \Dom S\cap\Dom T$, which yields the inclusion $\Dom(S+T)^\pi_0 \subset \Dom S^\pi_0 \cap \Dom T^\pi_0$. Taking closures then proves the statement. 
\item We have $\Dom(ST) = \{\psi\in E : T\psi\in\Dom S\}$. 
If $\xi = \sum_n\psi_n\hot h_n\in\Dom(ST)^\pi_0 := \Dom(ST)\odot\mH^\pi$, then 
$$
T^\pi\xi = \sum_n (T\psi_n)\hot h_n \in \Dom S\hodot\mH^\pi = \Dom S^\pi_0 \subset \Dom S^\pi .
$$
Hence we have $\Dom(ST)^\pi_0 \subset \Dom(S^\pi T^\pi)$. Taking closures then proves the statement. \qedhere
\end{enumerate}
\end{proof}

\begin{lem}
\label{lem:approx_id_loc}
Let $\D$ be a regular symmetric operator on a Hilbert $B$-module $E$, and let $\pi\colon B\to\B(\mH^\pi)$ be a $*$-representation of $B$. If $\{\phi_k\}_{k\in\N}\subset\End_B(E)$ is an adequate approximate identity for $\D$, then $\{\phi_k^\pi\}_{k\in\N}$ is an adequate approximate identity for $\D^\pi$. 
\end{lem}
\begin{proof}
By \cref{lem:loc_sum_comp} we have $[\D,\phi_k]^\pi \subset \bar{[\D^\pi,\phi_k^\pi]}$. But $[\D,\phi_k]$ is bounded, so we must have $[\D,\phi_k]^\pi = \bar{[\D^\pi,\phi_k^\pi]}$, and therefore $\sup_{k\in\N}\|\bar{[\D^\pi,\phi_k^\pi]}\| < \infty$. 
Lastly, we observe that $\phi_k^\pi \cdot \Dom (\D^*)_0^\pi \subset \Dom \D^\pi$, but it remains to check that $\phi_k^\pi\cdot\Dom(\D^\pi)^* \subset \Dom \D^\pi$. 

So, let $\psi\in\Dom(\D^*)^\pi$, and take $\psi_n\in\Dom(\D^*)^\pi_0$ such that $\psi_n\to\psi$ in the graph norm of $(\D^*)^\pi$. Since $\D$ is regular, we know from the local-global principle that $(\D^*)^\pi = (\D^\pi)^*$. 
We then have the equality
$$
\D^\pi \phi_k^\pi \psi_n = (\D^\pi)^* \phi_k^\pi \psi_n = \phi_k^\pi (\D^\pi)^* \psi_n + \big[(\D^\pi)^*,\phi_k^\pi\big] \psi_n .
$$
From \cref{lem:comm_D*} we know that $\bar{\big[(\D^\pi)^*,\phi_k^\pi\big]} = \bar{\big[\D^\pi,\phi_k^\pi\big]}$, which is bounded. By assumption, $(\D^\pi)^* \psi_n$ converges to $(\D^\pi)^* \psi$. Hence $\D^\pi \phi_k^\pi \psi_n$ also converges, which shows that $\phi_k^\pi \psi$ lies in the domain of $\D^\pi$. 
\end{proof}

\begin{thm}
\label{thm:approx_id_reg-sa}
Let $\D$ be a regular symmetric operator on a Hilbert $B$-module $E$. Then $\D$ is self-adjoint if and only if there exists an adequate approximate identity for $\D$. 
\end{thm}
\begin{proof}
If $\D$ is self-adjoint, then $\phi_k := (1+\frac1k\D^2)^{-\frac12}$ (for $0<k\in\N$) gives an adequate approximate identity for $\D$. Conversely, suppose there exists an adequate approximate identity for $\D$. For any representation $(\pi,\mH^\pi)$ of $B$, we obtain by \cref{lem:approx_id_loc} an adequate approximate identity for the localisation $\D^\pi$. By \cref{prop:approx_id_sa}, this implies that the operator $\D^\pi$ is self-adjoint. The local-global principle (\cref{thm:loc-glob}) then shows that $\D$ is self-adjoint. 
\end{proof}

\begin{remark}
We emphasise that the existence of an adequate approximate identity cannot be used to show that a symmetric operator must be regular, because \cref{lem:approx_id_loc} relies on the \emph{assumption} of regularity. In practice, if one does not (yet) know if a symmetric operator is regular, it can be more fruitful to try to apply \cref{prop:approx_id_sa} to the localisations of the symmetric operator, and then employ the local-global principle. Indeed, this is the approach we will use in the following section. 
\end{remark}

\section{Locally bounded perturbations}
\label{sec:loc_bdd}

\begin{defn}
Let $M$ be a densely defined operator on $E$ and let $\{\phi_k\}_{k\in\N}\subset\End_B(E)$ be a sequential approximate identity. We say that $M$ is \emph{locally bounded (with respect to $\{\phi_k\}$)} if $\phi_k\cdot\Dom M\to\Dom M$ and $M\phi_k$ is bounded on $\Dom M$ (for all $k\in\N$). 
\end{defn}

\begin{lem}
\label{lem:Dom_M}
Let $M$ be a semi-regular operator on $E$, and let $\phi=\phi^*\in\End_B(E)$ be such that $\phi\cdot\Dom M\subset\Dom M$ and $M\phi$ is bounded on $\Dom M$. Then: 
\begin{enumerate}
\item $\phi\cdot E\subset\Dom\bar M$ and $\bar{M\phi} = \bar M\phi$;
\item $\bar M\phi$ is adjointable, and its adjoint equals the closure of $\phi M^*$;
\item if $M$ is symmetric, then the commutator $[M,\phi]$ is bounded and its closure equals $\bar{M\phi} - (M\phi)^*$. 
\end{enumerate}
\end{lem}
\begin{proof}
\begin{enumerate}
\item Let $\eta\in E$. 
Since $\Dom M$ is dense in $E$, there exist $\eta_n\in\Dom M$ such that $\eta_n\to\eta$. Then we have $\phi\eta_n\to \phi\eta$ and 
$$
M(\phi\eta_n) = (M\phi)\eta_n = \bar{M\phi}\eta_n \to \bar{M\phi}\eta ,
$$
where we used that $M\phi$ is bounded on $\Dom M$ and hence its closure $\bar{M\phi}$ is bounded on $E$. This shows that $\phi\eta$ lies in the domain of the closure of $M$, and we have $\bar M\phi\eta = \bar{M\phi}\eta$. 
\item For $\psi\in E$ and $\xi\in\Dom M^*$ we see that $\la\xi|\bar M\phi\psi\ra = \la\phi M^*\xi|\psi\ra$, which shows that $\bar M\phi$ has a densely defined adjoint $\phi M^*$. Since 
$$
\|\phi M^*\xi\| = \sup_{\|\psi\|=1} \{\|\la\phi M^*\xi|\psi\ra\|\} = \sup_{\|\psi\|=1} \{\|\la\xi|\bar M\phi\psi\ra\|\} \leq \|\xi\| \|\bar M\phi\| ,
$$
we also see that $\phi M^*$ is bounded, and hence that $\phi M^*$ extends to a bounded operator which is the adjoint of $\bar M\phi$ (in particular, $\bar M\phi$ is adjointable). 
\item The commutator $[M,\phi]$ is densely defined on $\Dom M$, and for $\psi\in\Dom M$ we have 
$$
[M,\phi]\psi = M\phi\psi - \phi M\psi = M\phi\psi - \phi M^*\psi = (M\phi - (M\phi)^*) \psi ,
$$
where we have used the symmetry of $M$. Since $M\phi$ is bounded, so is $[M,\phi]$. \qedhere
\end{enumerate}
\end{proof}

We will be considering perturbations of a self-adjoint operator $\D$ by a locally bounded operator $M$. We start with a lemma which allows us to control the domain of the adjoint of the perturbed operator. 

\begin{lem}
\label{lem:loc_bdd_pert_dom}
Let $\D$ and $M$ be symmetric operators on $E$ such that $\Dom\D\cap\Dom M$ is dense. Let $\phi\in\End_B(E)$ be a self-adjoint operator such that
\begin{enumerate}
\item $\phi\cdot\Dom\D\subset\Dom\bar\D$ and $[\bar\D,\phi]$ is bounded on $\Dom\D$;
\item $\phi\cdot\Dom M\subset\Dom M$, and $M\phi$ is bounded on $\Dom M$.
\end{enumerate}
Then $\phi\cdot\Dom(\bar\D+\bar M)^* \subset \Dom\D^*\cap\Dom\bar M$. 
\end{lem}
\begin{proof}
Let $\xi\in\Dom(\bar\D+\bar M)^*$. 
We know from \cref{lem:Dom_M} that $\phi\xi\in\Dom\bar M$ and $\bar{M\phi} = \bar M\phi$. 
For $\psi\in\Dom\D$, we calculate
\begin{align*}
\la \phi\xi|\D\psi\ra &= \la\xi|\phi\D\psi\ra = \la\xi|(\bar\D+\bar M)\phi\psi\ra - \la\xi|\bar{M\phi}\psi\ra - \la\xi|[\bar\D,\phi]\psi\ra \\
&= \la \phi(\bar\D+\bar M)^*\xi|\psi\ra - \la(M\phi)^*\xi|\psi\ra - \la[\bar\D,\phi]^*\xi|\psi\ra .
\end{align*}
Since these equalities hold for all $\psi\in\Dom\D$, we conclude that $\phi\xi$ lies in the domain of $\D^*$. 
\end{proof}

Similarly to \cref{lem:approx_id_loc}, we prove next that an adequate approximate identity for $M$ also yields an adequate approximate identity for the localisation $M^\pi$. In this case however, thanks to the local boundedness of $M$, we do not need to assume that $M$ is regular. 
\begin{lem}
\label{lem:approx_id_loc_M}
Let $M$ be a symmetric operator on a Hilbert $B$-module $E$, and let $\pi\colon B\to\B(\mH^\pi)$ be a $*$-representation of $B$. 
If $\{\phi_k\}_{k\in\N}\subset\End_B(E)$ is an adequate approximate identity for $M$ and $M$ is locally bounded w.r.t.\ $\{\phi_k\}$, then $\{\phi_k^\pi\}_{k\in\N}$ is an adequate approximate identity for $M^\pi$ and $M^\pi$ is locally bounded w.r.t.\ $\{\phi_k^\pi\}$. 
\end{lem}
\begin{proof}
First, we note that $\phi_k^\pi\cdot\Dom M^\pi_0\subset\Dom M^\pi_0$. Next, for $\psi\in\Dom M^\pi$ there exist $\psi_n\in\Dom M^\pi_0$ such that $\psi_n\to\psi$ in the graph norm of $M^\pi$. Then $\phi_k^\pi\psi_n\to\phi_k^\pi\psi$ in norm (because $\phi_k^\pi$ is bounded) and
$$
M^\pi\phi_k^\pi \psi_n = (M\phi_k)^\pi \psi_n \to (M\phi_k)^\pi \psi ,
$$
because $(M\phi_k)^\pi$ is bounded. Hence $\phi_k^\pi\psi\in\Dom M^\pi$, and we have the equality $M^\pi\phi_k^\pi\psi = (M\phi_k)^\pi\psi$, which shows that $M^\pi$ is locally bounded w.r.t.\ $\{\phi_k^\pi\}$. By \cref{lem:Dom_M} this implies that $\phi_k^\pi\cdot\Dom(M^\pi)^* \subset \Dom M^\pi$. Lastly, from \cref{lem:loc_sum_comp} we know that $[M,\phi_k]^\pi \subset \bar{[M^\pi,\phi_k^\pi]}$. Since $[M,\phi_k]^\pi$ is bounded, we have $[M,\phi_k]^\pi = \bar{[M^\pi,\phi_k^\pi]}$ and hence that $\sup_{k\in\N}\|\bar{[M^\pi,\phi_k^\pi]}\| \leq \sup_{k\in\N}\|[M,\phi_k]\| < \infty$. Thus $\{\phi_k^\pi\}_{k\in\N}$ is an adequate approximate identity for $M^\pi$. 
\end{proof}

We are now ready to prove that, if we have a suitable approximate identity $\{\phi_k\}_{k\in\N}$, then the perturbation of a regular self-adjoint operator $\D$ by a locally bounded symmetric operator $M$ is again regular self-adjoint. Apart from local boundedness, the only additional assumption is that the commutators $[M,\phi_k]$ are uniformly bounded. 

\begin{thm}
\label{thm:reg_sa_sum_loc_bdd}
Let $E$ be a Hilbert $B$-module. 
Let $\D$ be a regular self-adjoint operator on $E$ and let $M$ be a symmetric operator on $E$ such that $\Dom\D\cap\Dom M$ is dense. 
Let $\{\phi_k\}_{k\in\N}\subset\End_B(E)$ be an adequate approximate identity for $\D$, such that $M$ is locally bounded (w.r.t.\ $\{\phi_k\}$) and $\sup_{k\in\N}\|[M,\phi_k]\| < \infty$. 
Then $\bar{\D+\bar M}$ is regular and self-adjoint. 

Furthermore, if $\phi_k\cdot\Dom\D\subset\Dom M$ for each $k\in\N$, then in fact $\bar{\D+M}$ is regular and self-adjoint (and therefore equal to $\bar{\D+\bar M}$). 
\end{thm}
\begin{proof}
Consider a representation $(\pi,\mH^\pi)$ of $B$. We will first show that $\{\phi_k^\pi\}_{k\in\N}$ is an adequate approximate identity for the symmetric operator $\D^\pi+M^\pi$, so that $\D^\pi+M^\pi$ is essentially self-adjoint. 

By \cref{lem:approx_id_loc,lem:approx_id_loc_M}, $\{\phi_k^\pi\}_{k\in\N}$ is an adequate approximate identity for both $\D^\pi$ and $M^\pi$. It then follows immediately that $[\D^\pi+M^\pi,\phi_k]$ is uniformly bounded on $\Dom(\D^\pi+M^\pi)$. We know from \cref{lem:comm_D*} that then $[\bar{\D^\pi+M^\pi},\phi_k]$ is also bounded on $\Dom\bar{\D^\pi+M^\pi}$, and that $\bar{[\bar{\D^\pi+M^\pi},\phi_k]} = \bar{[\D^\pi+M^\pi,\phi_k]}$. 
Since $(\D^\pi)^* = \D^\pi$, it follows from \cref{lem:loc_bdd_pert_dom} that $\phi_k^\pi\cdot\Dom(\D^\pi+M^\pi)^* \subset \Dom\D^\pi\cap\Dom M^\pi \subset \Dom\bar{\D^\pi+M^\pi}$. Hence $\D^\pi+M^\pi$ is essentially self-adjoint by \cref{prop:approx_id_sa}. 

Suppose that we have the inclusion $\phi_k\cdot\Dom\D\subset\Dom M$. 
In order to conclude from the local-global principle that $\bar{\D+M}$ is regular and self-adjoint, we need to know that the localisations $(\bar{\D+M})^\pi$ are self-adjoint. Therefore it remains to show that $(\bar{\D+M})^\pi = \bar{\D^\pi+M^\pi}$.
By \cref{lem:loc_sum_comp}, it is sufficient to show that $\D^\pi + M^\pi \subset (\D+M)^\pi$.  

Let $\eta \in \Dom(\D^\pi+M^\pi) = \Dom\D^\pi \cap \Dom M^\pi$. 
Since 
$\Dom\D\hodot_B\mH^\pi$ is a core for $\D^\pi$, 
there exists a sequence $\eta_n\in\Dom\D\hodot_B\mH^\pi$ such that $\eta_n\to\eta$ and $(\D\hot1)\eta_n\to\D^\pi\eta$. 
First, we will check that $\phi_k^\pi\eta \in \Dom(\D+M)^\pi$. 
Because $\phi_k\cdot\Dom\D \subset \Dom\D\cap\Dom M$, we have 
\begin{align*}
(\D+M)^\pi \phi_k^\pi\eta_n &= \big((\D+M)\phi_k\hot1\big)\eta_n
= \big((\phi_k\D+[\D,\phi_k]+M\phi_k)\hot1\big) \eta_n \\
&= \phi_k^\pi (\D\hot1) \eta_n + [\D,\phi_k]^\pi \eta_n + (M\phi_k)^\pi \eta_n \\
&\xrightarrow{n\to\infty} \phi_k^\pi \D^\pi\eta + [\D,\phi_k]^\pi \eta + (M\phi_k)^\pi \eta 
= (\D^\pi + M^\pi) \phi_k^\pi \eta ,
\end{align*}
where on the last line we used that $[\D,\phi_k]^\pi = [\D^\pi,\phi_k^\pi]$ (from the proof of \cref{lem:approx_id_loc}) and that $(M\phi_k)^\pi = M^\pi\phi_k^\pi$. 
Hence $\phi_k^\pi \eta$ is an element in $\Dom(\D+M)^\pi$. 

Second, we observe that we have the convergences $\phi_k^\pi \eta\to\eta$ and 
$$
(\D+M)^\pi \phi_k^\pi \eta = (\D^\pi + M^\pi) \phi_k^\pi \eta = \phi_k^\pi (\D^\pi + M^\pi) \eta + [\D^\pi + M^\pi, \phi_k^\pi] \eta .
$$
Since $\{\phi_k^\pi\}_{k\in\N}$ is an adequate approximate identity for both $\D^\pi$ and $M^\pi$, $(\D+M)^\pi \phi_k^\pi \eta$ is a bounded sequence, and therefore $\eta \in \Dom(\D+M)^\pi$ by \cref{lem:op_closure}. Thus we have shown that $(\bar{\D+M})^\pi = \bar{\D^\pi+M^\pi}$ is self-adjoint, and the local-global principle (\cref{thm:loc-glob}) then tells us that $\bar{\D+M}$ is regular and self-adjoint. 

If we do not have the inclusion $\phi_k\cdot\Dom\D\subset\Dom M$, then we nevertheless have $\phi_k\cdot\Dom\D\subset\Dom\bar M$ by \cref{lem:Dom_M}. Hence the proof given above applies to $\bar M$ instead of $M$, and we conclude that $\bar{\D+\bar M}$ is regular and self-adjoint. 
\end{proof}


\begin{example}
\label{eg:potential}
\begin{enumerate}
\item 
Let $\bV$ be a hermitian vector bundle over a complete Riemannian manifold $X$. Let $\D$ be a symmetric first-order differential operator with initial domain $\Gamma_c^\infty(\bV)$, and suppose that $\D$ has bounded propagation speed. 
Since the manifold is complete, there exist functions $\phi_k\in C_0(X,\R)$ (for $k\in\N$), converging pointwise to $1$, such that $\sup_{x\in X}\|d\phi_k(x)\| \to 0$.%
\footnote{For instance, given a smooth proper function $\rho\colon X\to\R$ with uniformly bounded gradient (e.g.\ a smooth approximation of the distance function $x\mapsto d(x,x_0)$ for some $x_0\in X$), choose a cutoff function $\chi\in C_0^\infty(\R)$ such that $0\leq\chi\leq1$, $\chi=1$ near $0$, and $|\chi'|\leq1$, and then define $\phi_k(x) := \chi(\frac1k\rho(x))$. }
Since $\D$ has bounded propagation speed, the sequence $\{\phi_k\}_{k\in\N}$ forms an adequate approximate identity for $\D$. Hence $\D$ is essentially self-adjoint. 

Any (continuous) symmetric endomorphism $T\in\Gamma(\End\bV)$ is locally bounded and we have $[T,\phi_k]=0$. Hence $\bar{\D+T}$ is also self-adjoint. 
In fact, it is not even necessary for $T$ to be continuous, as long as it is locally bounded; the same result therefore holds for a symmetric endomorphism $T\in L^\infty_{\textnormal{loc}}(\End\bV)$, i.e.\ if $\esssup_{x\in K}\|T(x)\| < \infty$ for any compact subset $K\subset X$. 

\item 
In the above example, it is not necessary that the vector bundle $\bV$ has finite rank. Consider for instance the following setup. 
Let $E$ be a countably generated Hilbert module over a $\sigma$-unital $C^*$-algebra $B$, and let $T\in C(X,\End_B(E))$ be a symmetric operator on the Hilbert $C_0(X,B)$-module $C_0(X,E)$, which is densely defined on the domain $C_c(X,E)$. Then $T$ is locally bounded, and commutes with any approximate identity $\phi_k\in C_0(X,\R)$. 

Given $\D$ on $\bV\to X$ as above, we consider the Hilbert $B$-module $L^2(X,E\otimes\bV) := C_0(X,E) \otimes_{C_0(X)} L^2(X,\bV)$. The operator $\D$ on $L^2(X,\bV)$ extends to a regular self-adjoint operator $1\otimes_d\D$ on $L^2(X,E\otimes\bV)$ given by 
\[
(1\otimes_d\D)(\xi\otimes\psi) := \xi\otimes\D\psi + (1\otimes\sigma)(d\xi)\psi ,
\]
for any $\xi\in C_c^\infty(X,E)$ and $\psi\in\Gamma_c^\infty(X,\bV)$, where $\sigma$ denotes the principal symbol of $\D$ (for more details, see \cite{KL13}, where this operator is called $1\otimes_{\nabla^{\textnormal{Gr}}}\D$). By \cref{thm:reg_sa_sum_loc_bdd}, given any adequate approximate identity $\phi_k\in C_c(X,\R)$ for $\D$, the closure of the operator $1\otimes_d\D + T\otimes1$ is also regular self-adjoint. 
\end{enumerate}
\end{example}

\section{Stability of unbounded Kasparov modules}
\label{sec:Kasp_mod_stab}

In the previous section, we proved that perturbations of regular self-adjoint operators by `locally bounded' operators are again regular self-adjoint. In this section we apply this result to noncommutative geometry \cite{Connes94} and unbounded $\KK$-theory \cite{Kas80b,BJ83}. 
More precisely, we will show that the class of an unbounded Kasparov module is stable under locally bounded perturbations. Throughout the remainder of this article, we will assume that $A$ and $B$ are $\Z_2$-graded $C^*$-algebras such that $A$ is separable and $B$ is $\sigma$-unital. 

\begin{defn}[\cite{BJ83}]
\label{defn:Kasp_mod}
An (even) \emph{unbounded Kasparov $A$-$B$-module} $(\A,{}_{\pi}E_B,\D)$ is given by a $\Z_2$-graded, countably generated, right Hilbert $B$-module $E$, a $\Z_2$-graded $*$-homomorphism $\pi\colon A\to\End_B(E)$, a separable dense $*$-subalgebra $\A\subset A$, and a regular self-adjoint odd operator $\D\colon\Dom\D\subset E\to E$ such that 
\begin{enumerate}
\item we have the inclusion $\pi(\A)\cdot\Dom\D\subset\Dom\D$, and the graded commutator $[\D,\pi(a)]_\pm$ is bounded on $\Dom\D$ for each $a\in\A$; 
\item the resolvent of $\D$ is \emph{locally compact}, i.e.\ $\pi(a) (\D\pm i)^{-1}$ is compact for each $a\in A$.
\end{enumerate}
An \emph{odd} unbounded Kasparov $A$-$B$-module $(\A,{}_{\pi}E_B,\D)$ is defined in the same way, except that $A$, $B$, and $E$ are assumed to be trivially graded, and $\D$ is not required to be odd. 

The $*$-homomorphism $\pi\colon A\to\End_B(E)$ is called \emph{non-degenerate} (or \emph{essential}) if $\pi(A)\cdot E$ is dense in $E$. 
If no confusion arises, we will usually write $(\A,E_B,\D)$ instead of $(\A,{}_{\pi}E_B,\D)$ and $a$ instead of $\pi(a)$. 
If $B=\C$ and $A$ is trivially graded, we will write $E=\mH$ and refer to $(\A,\mH,\D)$ as an (even or odd) \emph{spectral triple} over $A$ (see \cite{Connes94}). 
\end{defn}

Given a Kasparov module $(\A,E_B,\D)$, we will consider an approximate identity $\{\phi_k\}_{k\in\N}\subset\A$. As in the previous section, we will consider a perturbation of $\D$ by a locally bounded symmetric operator $M$ on $E$. In fact, we will assume a slightly stronger version of local boundedness: we require not only that $M\phi_k$ is bounded but also that $Ma$ is bounded for all $a\in\A$. 

The following theorem only applies to unbounded Kasparov modules for which there exists an adequate approximate identity which lies in the algebra. We note that every Kasparov class can be represented by such a module \cite[Proposition 4.18]{MR16}. 
However, not every (naturally occurring) unbounded Kasparov module admits such an approximate identity 
(consider, for instance, a Dirac-type operator on a manifold with boundary). 

\begin{thm}
\label{thm:Kasp_stab_loc_bdd}
Let $(\A,E_B,\D)$ be an (even or odd) unbounded Kasparov $A$-$B$-module, such that the $*$-homomorphism $\pi\colon A\to\End_B(E)$ is non-degenerate. Let $M$ be a closed symmetric operator on $E$ with $\deg M = \deg\D$ such that $a\cdot\Dom M\subset\Dom M$ and $Ma$ is a bounded operator for all $a\in\A$. Let $\{\phi_k\}_{k\in\N}\subset\A$ be an adequate approximate identity for $\D$ such that $\sup_{k\in\N}\|[M,\phi_k]\| < \infty$. Then $(\A,E_B,\bar{\D+M})$ is also an unbounded Kasparov $A$-$B$-module, and it represents the same class as $(\A,E_B,\D)$. 
\end{thm}
\begin{proof}
The assumptions on $M$ imply that $M$ is locally bounded w.r.t.\ $\{\phi_k\}$. 
Since $M$ is closed, we know by \cref{lem:Dom_M} that $\A\cdot E\subset\Dom M$, so the intersection $\Dom\D\cap\Dom M$ contains the dense subset $\A\cdot\Dom\D$. Thus $\D+M$ is densely defined, and we know from \cref{thm:reg_sa_sum_loc_bdd} that $\bar{\D+M}$ is regular and self-adjoint. 
The commutator $[M,a] = Ma-aM$ equals $Ma-(Ma^*)^*$ (cf.\ \cref{lem:Dom_M}) and is therefore bounded for all $a\in\A$. 
It is then immediate that $\D+M$ has bounded commutators with $a\in\A$, and by \cref{lem:comm_D*} this implies that $\bar{\D+M}$ also has bounded commutators with $a\in\A$. 
We will show the local compactness of the resolvent of $\bar{\D+M}$. 
Let us write $\E := \Dom\D\cap\Dom M$ for the (initial) domain of $\D+M$. Since $\D+M$ is essentially self-adjoint, we know that $(\bar{\D+M}\pm i)\E$ is dense in $E$. For any $\psi\in(\bar{\D+M}\pm i)\E$ we have $(\bar{\D+M}\pm i)^{-1}\psi\in\E\subset\Dom\D$. For $a_1,a_2\in\A$, we can then rewrite
\begin{align*}
&a_1a_2 (\bar{\D+M}\pm i)^{-1} \psi = a_1a_2 (\D\pm i)^{-1} (\D\pm i) (\bar{\D+M}\pm i)^{-1} \psi \\
&= a_1 (\D\pm i)^{-1} a_2 (\D\pm i) (\bar{\D+M}\pm i)^{-1} \psi - a_1 \big[(\D\pm i)^{-1}, a_2\big] (\D\pm i) (\bar{\D+M}\pm i)^{-1} \psi \\
&= a_1 (\D\pm i)^{-1} \big(a_2 - a_2M (\bar{\D+M}\pm i)^{-1}\big) \psi + a_1 (\D\pm i)^{-1} \big[\D, a_2\big] (\bar{\D+M}\pm i)^{-1} \psi .
\end{align*}
Since we assumed that $a_1(\D\pm i)^{-1}$ is compact and since such $\psi$ are dense in $E$, it follows that $a_1a_2(\bar{\D+M}\pm i)^{-1}$ is compact. Because products $a_1a_2$ are dense in $A$, it then follows that $a(\bar{\D+M}\pm i)^{-1}$ is compact for all $a\in A$. Thus we have shown that $(\A,E_B,\bar{\D+M})$ is also an unbounded Kasparov $A$-$B$-module.

To prove that $(\A,E_B,\bar{\D+M})$ represents the same class as $(\A,E_B,\D)$, we will show that $(\A,E_B,\bar{\D+M})$ represents the Kasparov product of $1_A = [(A,0)] \in \KK(A,A)$ with $[(\A,E_B,\D)] \in \KK(A,B)$. 
For this purpose we need to check the three conditions in Kucerovsky's theorem 
\cite[Theorem 13]{Kuc97}.
Since we are considering the zero operator on $A$, the second and third of Kucerovsky's conditions are trivial. For the first condition we need to show that the commutator
$$
\left[ \mattwo{\bar{\D+M}}{0}{0}{\D} , \mattwo{0}{T_a}{T_a^*}{0} \right] 
$$
is bounded on $\Dom\bar{\D+M}\oplus\Dom\D$ for all $a$ in a dense subset of $A$. 

We have the isomorphism $A\otimes_AE \simeq E$. 
For $e\in E$, the operator $T_a\colon e\mapsto a\otimes e$ is then simply given by left multiplication with $a$, and its adjoint $T_a^*$ is left multiplication by $a^*$. We have $a\cdot\Dom\D\subset\Dom\D\cap\Dom M\subset\Dom\bar{\D+M}$ and $a^*\cdot(\Dom\D\cap\Dom M)\subset\Dom\D$, so that the following commutator is well-defined and equal to
$$
\left[ \mattwo{\D+M}{0}{0}{\D} , \mattwo{0}{a}{a^*}{0} \right] = \mattwo{0}{[\D,a]+Ma}{{[\D,a^*]}-a^*M}{0} ,
$$
which is bounded on $(\Dom\D\cap\Dom M)\oplus\Dom\D$ for all $a\in\A\subset A$. Since $\Dom\D\cap\Dom M$ is a core for $\bar{\D+M}$, this commutator in fact extends to a bounded operator on $\Dom\bar{\D+M}\oplus\Dom\D$ (see \cref{lem:comm_D*}). Hence Kucerovsky's first condition holds. 
\end{proof}

\begin{example}
\label{eg:classical}
Let $\D$ be as in \cref{eg:potential}. If $\D$ is furthermore elliptic, then $(C_c^\infty(X), L^2(\bV), \bar\D)$ is a spectral triple. 
Then for any symmetric $T\in L^\infty_{\textnormal{loc}}(\End\bV)$ we have that $(C_c^\infty(X), L^2(\bV), \bar{\D+T})$ is also a spectral triple which represents the same K-homology class. 

More concretely, consider the standard odd spectral triple $(C_c^\infty(\R), L^2(\R), \bar{i\partial_x})$ over the real line. 
For any real-valued $f\in L^\infty_{\textnormal{loc}}(\R)$, denote by $M_f$ the operator of multiplication by $f$ on the Hilbert space $L^2(\R)$. We then find that the odd spectral triple $(C_c^\infty(\R), L^2(\R), \bar{i\partial_x+M_f})$ represents the same K-homology class as the standard odd spectral triple $(C_c^\infty(\R), L^2(\R), \bar{i\partial_x})$. This generalises a previous result in \cite[\S5.3]{vdDPR13}, where the equivalence of these spectral triples was shown for the special case $f(x)=x$. 
\end{example}

\section{Odd KK-theory}
\label{sec:odd_KK}

In this section we will apply \cref{thm:Kasp_stab_loc_bdd} to the odd version(s) of unbounded $\KK$-theory, but first we shall have a look at the bounded case. 
For trivially graded $C^*$-algebras $A$ and $B$, there are two types of (bounded) representatives for a class in the odd $\KK$-theory $\KK^1(A,B) = \KK(A\otimes\CCliff_1,B)$:
\begin{enumerate}
\item an \emph{odd} Kasparov $A$-$B$-module $(\A,E_B,F)$ (where the Hilbert module $E$ is trivially graded);
\item an (even) Kasparov $A\otimes\CCliff_1$-$B$-module $(\A\otimes\CCliff_1,\til E_B,\til F)$.\footnote{In fact, a third option is to consider a Kasparov $A$-$B\otimes\CCliff_1$-module, which is equivalent by the periodicity of the $\KK$-groups.}
\end{enumerate}
These two perspectives are equivalent 
(see \cite[Prop.\ IV.A.13]{Connes94}), 
which can be shown as follows. 
Given an odd module $(\A,{}_\pi E_B,F)$, one can construct an (even) Kasparov $A\otimes\CCliff_1$-$B$-module $(\A\otimes\CCliff_1,{}_{\til\pi}\til E_B,\til F)$ by setting
\begin{align}
\label{eq:doubling}
\til E &= E\oplus E , & \til\pi &= \pi\oplus\pi , & \til F &= \mattwo{0}{-iF}{iF}{0} , & \gamma &= \mattwo{1}{0}{0}{-1} , & e &= \mattwo{0}{1}{1}{0} ,
\end{align}
where $\gamma$ is the grading operator on $\til E$, and $e$ denotes the generator of $\CCliff_1$. We observe that the operator $F$ anti-commutes with $e$. 
Conversely, given an (even) Kasparov $A\otimes\CCliff_1$-$B$-module $(\A\otimes\CCliff_1,{}_{\til\pi}\til E_B,\til F)$, the graded Hilbert module $\til E$ decomposes as $\til E^+\oplus\til E^-$, and we may identify $E:=\til E^+$ with $\til E^-$ using the Clifford generator $e\in\CCliff_1$. 
Thus, up to unitary equivalence, this Kasparov module is of the form
\begin{align}
\label{eq:even_mod_Cliff}
\til E &= E\oplus E , & \til\pi &= \pi\oplus\pi , & \til F &= \mattwo{0}{F_-}{F_+}{0} , & \gamma &= \mattwo{1}{0}{0}{-1} , & e &= \mattwo{0}{1}{1}{0} ,
\end{align}
where $E$ is a Hilbert $B$-module with a $*$-homomorphism $\pi\colon A\to\End_B(E)$, $\gamma$ is the grading operator on $\til E$, and $e$ is the generator of $\CCliff_1$. We point out that the operator $\til F$ in general does not commute with the Clifford generator. 
Since the graded commutators of $\til F$ with the algebra $A\otimes\CCliff_1$ are compact, one finds that $[F_\pm,a]$ and $a(F_++F_-)$ are compact for all $a\in A$. In particular, this implies that $\frac12a(\til F+e\til Fe)$ is compact. 
Since $a(\til F-\til F^*)$ is compact, we also know that $a(F_\pm-F_\mp^*)$ is compact. 
Using these properties, one can check that $\til F':=\frac12(\til F-e\til Fe)$ also yields a Kasparov module which is 
a locally compact perturbation of $\til F$, and therefore (see \cite[Proposition 17.2.5]{Blackadar98}) $\til F'$ is operator-homotopic to $\til F$. 
Then the operator $F := -\frac i2(F_+-F_-)$ yields an \emph{odd} Kasparov $A$-$B$-module $(\A,{}_\pi E_B,F)$ which represents the same class. 

Let us now consider the case of unbounded representatives. Again, 
we consider two types of unbounded representatives for a class in the odd $\KK$-theory $\KK^1(A,B) = \KK(A\otimes\CCliff_1,B)$:
\begin{enumerate}
\item an \emph{odd} unbounded Kasparov $A$-$B$-module $(\A,E_B,\D)$ (where the Hilbert module $E$ is trivially graded);
\item an (even) unbounded Kasparov $A\otimes\CCliff_1$-$B$-module $(\A\otimes\CCliff_1,\til E_B,\til\D)$.
\end{enumerate}
Of course, these two perspectives are again equivalent, because their bounded transforms are equivalent. However, the question remains whether there is a natural, canonical way of implementing this equivalence while working only in the unbounded picture (i.e.\ without using the bounded transform). 
From an odd unbounded Kasparov $A$-$B$-module one constructs an (even) unbounded Kasparov $A\otimes\CCliff_1$-$B$-module as in \cref{eq:doubling}. Conversely, given an unbounded Kasparov $A\otimes\CCliff_1$-$B$-module $(\A\otimes\CCliff_1,\til E_B,\til\D)$, one would need to show that $\til\D':=\frac12(\til\D-e\til\D e)$ represents the same Kasparov class. However, in general it is not even clear if $\til\D'$ is regular self-adjoint. 
In this section we will prove that $\til\D'$ is regular self-adjoint and represents the same class as $\til\D$ whenever there exists an adequate approximate identity. 

Let us fix our notation. 
Let $(\A\otimes\CCliff_1,\til E_B,\til\D)$ be an (even) unbounded Kasparov $A\otimes\CCliff_1$-$B$-module, such that the $*$-homomorphism $\til\pi\colon A\to\End_B(\til E)$ is non-degenerate and commutes with the action of $\CCliff_1$. 
As in \cref{eq:even_mod_Cliff}, this Kasparov module is (up to unitary equivalence) of the form
\begin{align*}
\til E &= E\oplus E , & \til\pi &= \pi\oplus\pi , & \til\D &= \mattwo{0}{\D_-}{\D_+}{0} , & \gamma &= \mattwo{1}{0}{0}{-1} , & e &= \mattwo{0}{1}{1}{0} ,
\end{align*}
where $\Dom\til\D = \Dom\D_+\oplus\Dom\D_-$ and $\D_\pm^* = \D_\mp$. 
The operator $\til\D$ does not need to anti-commute with $e$ (we only know that $\til\D$ has \emph{bounded} graded commutators with the algebra). 
On the domain $\Dom\D_+\cap\Dom\D_-$ we define the operators
\begin{align}
\label{eq:D-M}
\D &:= - \frac i2 (\D_+-\D_-) , & M &:= \frac12 (\D_++\D_-) .
\end{align}
On $\E := (\Dom\D_+\cap\Dom\D_-)^{\oplus2}$ we define symmetric operators $\til\D'$ and $\til M$ by
\begin{align}
\label{eq:tilD'}
\til\D' &:= \mattwo{0}{-i\D}{i\D}{0} , & \til M &:= \mattwo{0}{M}{M}{0} , & \til\D|_\E &= \til\D'+ \til M = \mattwo{0}{-i\D+M}{i\D+M}{0} .
\end{align}
If we wish to reduce this module to an \emph{odd} unbounded Kasparov $A$-$B$-module $(\A,E_B,\D)$, we need to show that we can remove the operator $M$, without changing the underlying class in $\KK$-theory. 
If the algebra $\A$ is unital, then the assumption that the anti-commutator $[\til\D,1\otimes e]_\pm$ is bounded implies that $\til M$ is bounded. Hence $\til\D'$ is only a bounded perturbation of $\til\D$, and we know that unbounded Kasparov modules are stable under bounded perturbations. However, if $\A$ is non-unital, the operator $\til M$ can be unbounded. Nevertheless, similar reasoning shows that $\til M$ must be \emph{locally} bounded. With this observation, the following result is a straightforward consequence of \cref{thm:Kasp_stab_loc_bdd}.

\begin{thm}
\label{thm:Kasp_Cliff_anti-comm}
Let $A$ and $B$ be trivially graded $C^*$-algebras. Consider an (even) unbounded Kasparov $A\otimes\CCliff_1$-$B$-module $(\A\otimes\CCliff_1,\til E_B,\til\D)$, such that the $*$-homo\-mor\-phism $\til\pi\colon A\to\End_B(\til E)$ is non-degenerate and 
commutes with the action of $\CCliff_1$. 
Suppose that $\A$ contains an adequate approximate identity for $\til\D$. 
Then the operator $\til\D'$ defined in \cref{eq:tilD'} also yields an unbounded Kasparov module $(\A\otimes\CCliff_1,\til E_B,\bar{\til\D'})$ which represents the same Kasparov class as $(\A\otimes\CCliff_1,\til E_B,\til\D)$. 
\end{thm}
\begin{proof}
First note that for all $a$ in a dense subalgebra $\A\subset A$ we have $(a\otimes 1)\cdot\Dom\til\D\subset\Dom\til\D$ and $(a\otimes e)\cdot\Dom\til\D\subset\Dom\til\D$, which implies that $a\cdot\Dom\D_\pm\subset\Dom\D_+\cap\Dom\D_-$. Since we assumed that the $*$-homomorphism $\pi\colon A\to\End_B(E)$ is non-degenerate, the subset $\A\cdot\Dom\D_\pm$ is dense in $E$, and it follows that $\Dom\D_+\cap\Dom\D_-$ is also dense. Hence $\til\D'$ and $\til M$ are densely defined on the domain $(\Dom\D_+\cap\Dom\D_-)^{\oplus2}$, and they are symmetric because $\D_\pm^* = \D_\mp$. 

The graded commutators $[\til\D,a\otimes1]_\pm$ and $[\til\D,a\otimes e]_\pm$ are bounded for all $a\in\A$. 
The first commutator equals 
$$
[\til\D,a\otimes1]_\pm = \left[ \mattwo{0}{\D_-}{\D_+}{0} , \mattwo{a}{0}{0}{a} \right] = \mattwo{0}{{[\D_-,a]}}{{[\D_+,a]}}{0} , 
$$
which shows that the commutators $[\D_\pm,a]$ are bounded for all $a\in\A$. Next, we have the anti-commutator
\begin{align*}
[\til\D,a\otimes e]_\pm &= \left\{ \mattwo{0}{\D_-}{\D_+}{0} , \mattwo{0}{a}{a}{0} \right\} = \mattwo{\D_-a+a\D_+}{0}{0}{\D_+a+a\D_-} 
= \mattwo{2Ma-[\D_+,a]}{0}{0}{2Ma-[\D_-,a]} ,
\end{align*}
which shows that furthermore $Ma$ is bounded for all $a\in\A$, and hence $\til Ma$ is also bounded for all $a\in\A$. 

By assumption there exists an adequate approximate identity $\{\phi_k\}_{k\in\N}\subset\A$ for $\til\D$. The uniform bound $\sup_{k\in\N}\|[\til\D,\phi_k]\|<\infty$ implies that $\sup_{k\in\N}\|[\D_\pm,\phi_k]\|<\infty$ and hence that $\sup_{k\in\N}\|[\til M,\phi_k]\|<\infty$. 
From \cref{thm:Kasp_stab_loc_bdd} we then know that we have an unbounded Kasparov module $(\A\otimes\CCliff_1,\til E_B,\bar{\til\D-\bar{\til M}})$ which represents the same Kasparov class as $(\A\otimes\CCliff_1,\til E_B,\til\D)$. Finally, since $\phi_k\cdot\Dom\D_\pm\subset\Dom\D_+\cap\Dom\D_-$, we have $\phi_k\cdot\Dom\til\D\subset\Dom\til M$, so from \cref{thm:reg_sa_sum_loc_bdd} we know that $\bar{\til\D-\bar{\til M}} = \bar{\til\D-\til M} = \bar{\til\D'}$, which completes the proof. 
\end{proof}

\begin{coro}
With the same assumptions as in \cref{thm:Kasp_Cliff_anti-comm}, and with the operators $\D$ and $M$ defined as in \cref{eq:D-M}, we obtain two \emph{odd} unbounded Kasparov $A$-$B$-modules $(\A,E_B,\bar\D)$ and $(\A,E_B,\bar{\D+M})$, which both represent the same Kasparov class as $(\A\otimes\CCliff_1,\til E_B,\til\D)$. 
\end{coro}
\begin{proof}
For $(\A,E_B,\bar\D)$ the statement follows from \cref{thm:Kasp_Cliff_anti-comm} by observing that $\bar{\til\D'}$ anti-commutes with the Clifford generator, so we can restrict the even module $(\A\otimes\CCliff_1,\til E_B,\bar{\til\D'})$ to the odd module $(\A,E_B,\bar\D)$. For $(\A,E_B,\bar{\D+M})$ the statement follows again from \cref{thm:Kasp_stab_loc_bdd}.
\end{proof}

\begin{example}
Let $f\in L^\infty_{\textnormal{loc}}(\R)$ (see also \cref{eg:classical}). 
We then find that 
$$
\left( C_c^\infty(\R)\otimes\CCliff_1, L^2(\R) \oplus L^2(\R) , \bar\D = \mattwo{0}{\bar{\partial_x+M_f}}{\bar{-\partial_x+M_f}}{0} \right) 
$$
is an even spectral triple which 
represents the same K-homology class as the equivalent odd spectral triples $(C_c^\infty(\R), L^2(\R), \bar{i\partial_x+M_f})$ and $(C_c^\infty(\R), L^2(\R), \bar{i\partial_x})$. 
\end{example}

\section{Unbounded multipliers}
\label{sec:unbdd_mult}

A typical example of a locally bounded operator on $E$ would be an unbounded multiplier on a non-unital $C^*$-algebra $A\subset\End_B(E)$. 
In this section we will study this typical example in more detail. We will show that, given the existence of a suitable approximate identity, an unbounded Kasparov module is stable under perturbations by unbounded multipliers. In \cref{sec:cpt_res} we will apply this result to obtain an explicit construction of an unbounded multiplier such that the perturbed operator has compact resolvent. 

The typical case we have in mind is when the unbounded multiplier is even (e.g.\ if $A$ is trivially graded), which means we cannot use the unbounded multiplier as a perturbation of an odd operator (i.e., in an even unbounded Kasparov module). For this reason, we (initially) consider only \emph{odd} unbounded Kasparov modules. 

\begin{defn}
Let $A$ be a $C^*$-algebra. An \emph{unbounded multiplier} on $A$ is a linear map $m\colon\Dom m\to A$, where $\Dom m$ is a dense right ideal in $A$, which satisfies $m(ab) = (ma)b$ for all $a\in\Dom m$ and $b\in A$. An unbounded multiplier $m$ is called \emph{symmetric} if $(ma)^*b = a^*(mb)$ for all $a,b\in\Dom m$. 
\end{defn}

Let $E_B$ be a Hilbert $B$-module, and suppose we have a non-degenerate $*$-homomorphism $A\to\End_B(E)$. 
Then an unbounded multiplier $m$ on $A$ defines a densely defined operator $M$ on $E_B$ with initial domain $\Dom M := \Dom m\cdot E$ by $M(a\psi) := (ma)\psi$ (see \cite[Proposition 10.7 \& Lemma 10.8]{Lance95}). 
If $m$ is symmetric, then $M$ is also symmetric. 
By construction, $Ma$ is bounded for any $a\in\Dom m$.  
The following statement is then an immediate consequence of \cref{thm:Kasp_stab_loc_bdd}. 

\begin{coro}
\label{coro:unbdd_mult}
Let $(\A,E_B,\D)$ be an odd unbounded Kasparov $A$-$B$-module, such that the $*$-homomorphism $A\to\End_B(E)$ is non-degenerate. 
Let $m$ be a symmetric unbounded multiplier with $\A\subset\Dom m$, and denote by $M$ the corresponding operator on $E$. 
Suppose that $\A$ contains an adequate approximate identity $\{\phi_k\}_{k\in\N}$ for $\D$, such that $\sup_{k\in\N} \|[m,\phi_k]\| < \infty$. 
Then $(\A,E_B,\bar{\D+M})$ is also an unbounded Kasparov $A$-$B$-module, and it represents the same class as $(\A,E_B,\D)$. 
\end{coro}

\begin{remark}
In the above corollary, the only compatibility assumption between the approximate identity $\{\phi_k\}_{k\in\N}\subset\A$ and the unbounded multiplier $m$ is that $\sup_{k\in\N} \|[m,\phi_k]\| < \infty$. We note that if $\A$ is unital, every multiplier is in fact bounded and this assumption holds automatically. Furthermore, if $\A$ is non-unital but commutative, these commutators equal zero and the assumption therefore also holds automatically. Hence, this assumption is only relevant when the algebra $\A$ is both non-unital and non-commutative. 

Let us provide an example where this compatibility assumption fails.
Consider the $C^*$-algebra $C_0(\R)$. Let $m\in C(\R)$ be an unbounded multiplier, and let $\{\phi_k\}_{k\in\N}\subset\Dom m$ be an approximate identity. We consider the algebra of $2\times2$-matrices over $C_0(\R)$ with unbounded multiplier and approximate identity given by
\begin{align*}
\til m &:= \mattwo{0}{im}{-im}{0} , & \til\phi_k &:= \mattwo{1}{\frac1k}{\frac1k}{1} \phi_k .
\end{align*}
The commutator is then given by
$$
[\til m,\til\phi_k] = \frac{2i}k \mattwo{m}{0}{0}{-m} \phi_k .
$$
By choosing $m$ to approach infinity sufficiently fast, we can ensure that there is no uniform bound on $\{\frac1k m\phi_k\}_{k\in\N}$. Hence we see that the sequence $\{[\til m,\til\phi_k]\}_{k\in\N}$ in general need \emph{not} be uniformly bounded in $k$. 
\end{remark}

\subsection{Compactness of the resolvent}
\label{sec:cpt_res}

If $(\A,E_B,\D)$ is an unbounded Kasparov $A$-$B$-module for a non-unital $C^*$-algebra $A$, then in general the resolvent of $\D$ is only \emph{locally} compact. In practice, it can be much easier to deal with operators whose resolvent is in fact compact. 
In this section we address the following question: under which conditions can we find a locally bounded perturbation such that the perturbed operator has compact resolvent?
In fact, we will construct this locally bounded perturbation explicitly as an unbounded multiplier built from a given approximate identity. 

\noindent
\textbf{Standing Assumptions.}
Let $A$ and $B$ be trivially graded $C^*$-algebras, and suppose that $A$ is separable. 
Let $(\A,E_B,\D)$ be an odd unbounded Kasparov $A$-$B$-module, such that the representation $A\to\End_B(E)$ is non-degenerate. Let $\{\phi_k\}_{k\in\N}\subset\A$ be a commutative\footnote{I.e.\ $[\phi_k,\phi_m]=0$ for all $k,m\in\N$.} approximate unit for $A$ and an adequate approximate identity for $\D$. Without loss of generality, we assume we are given a countable total subset $\{a_j\}_{j\in\N}$ of $\A$ such that $\phi_k \in \Span\{a_j\}_{j\in\N}$ and $\|(\phi_{k+1}-\phi_k)a_j\| < 4^{-k}$ for all $j<k$. 

\begin{lem}
\label{lem:mult_vanish_inf}
The series
$m := \sum_{k\in\N} 2^k(\phi_{k+1}-\phi_k)$ gives a well-defined symmetric unbounded multiplier on $A$ such that $\A\cap\Dom m$ is dense in $A$, $[m,\phi_k] = 0$ (for all $k\in\N$), and $(\bar m\pm i)^{-1}$ lies in $A$. 
\end{lem}
\begin{proof}
Our argument roughly follows (part of) the proof of \cite[Theorem 1.25]{MR16}, to which we refer for more details. 
The unbounded multiplier $m$ is defined on 
\[
\Dom m := \Big\{ a\in A : \sum_{k\in\N} 2^k(\phi_{k+1}-\phi_k) a \text{ is norm-convergent in $A$} \Big\} .
\]
First one checks that $a_j\in\Dom m$, which shows that $m$ is densely defined, and in particular that $\A\cap\Dom m$ is dense in $A$. 
Since $\phi_k=\phi_k^*$, we know that $m$ is symmetric, and since $\{\phi_k\}_{k\in\N}$ is commutative, we have $[m,\phi_k]=0$. 

Consider the truncations $m_n := \sum_{k=1}^n 2^k(\phi_{k+1}-\phi_k) \in \A$. Let $B$ be the commutative $C^*$-algebra generated by $\{\phi_k\}_{k\in\N}$. Since $m_n\in B$, we also have $(m_n\pm i)^{-1}\in B$. Furthermore, the sequence $(m_n\pm i)^{-1}$ is strictly Cauchy, and therefore its limit $(\bar m\pm i)^{-1}$ lies in $M(B)$. 
By Gelfand-Naimark duality, there exists a locally compact Hausdorff space $X$ such that $B = C_0(X)$. 
Fix $0<t<1$, and consider the increasing sequence of compact sets $X_k := \{ x\in X : \phi_k(x) \geq t \}$ such that $X = \bigcup X_k$. For $x\in X\backslash X_k$ we have the inequality (see the proof of \cite[Theorem 1.25]{MR16})
\[
\sum_{n=0}^\infty 2^n (\phi_{n+1}(x)-\phi_n(x)) \geq (1-t) 2^k ,
\]
which shows that $(\bar m\pm i)^{-1} \in C_0(X) \subset A$. 
\end{proof}

The following lemma is a consequence of the closed graph theorem. A proof of this statement for Hilbert spaces can be found for instance in \cite[Lemma 8.4]{Sch12}. 
\begin{lem}
\label{lem:dom_inc_rel_bdd}
Let $S$ be a closed operator on a Hilbert $B$-module $E$, and let $T$ be a closable operator such that $\Dom S\subset\Dom T$. Then $T$ is relatively bounded by $S$. 
\end{lem}
\begin{proof}
We consider $\Dom S$ as a Hilbert module equipped with the graph norm of $S$, and we denote by $\bar T$ the closure of $T$. 
We will show that $T|_{\Dom S}\colon\Dom S\to E$ is closed. 
Consider a sequence $\psi_n\in\Dom S$ which converges to $\psi\in\Dom S$ (with respect to the graph norm of $S$) such that $T\psi_n$ converges in $E$. Since $T$ is closable (and $\psi_n\to\psi$ in $E$), we know that $T\psi_n$ converges to $\bar T\psi$. But $\psi\in\Dom S\subset\Dom T$, so $\bar T\psi =  T\psi$. 
Hence $T|_{\Dom S}\colon\Dom S\to E$ is a closed everywhere defined operator. 
The closed graph theorem then implies that $T|_{\Dom S}$ is bounded. 
\end{proof}

\begin{thm}
\label{thm:odd_Kasp_cpt_res}
Let $(\A,E_B,\D)$ and $\{\phi_k\}\subset\A$ be as in the Standing Assumptions. Suppose that 
$\|[\D,\phi_k]\| < 4^{-k}$ for all $k$. 
Let $M$ be the unbounded operator on $E$ corresponding to the unbounded multiplier $m := \sum_{k\in\N} 2^k(\phi_{k+1}-\phi_k)$. 
Write $\A_m := \A\cap\Dom m$. 
Then the operator
\[
\til\D' := \mattwo{0}{-i\D+\bar M}{i\D+\bar M}{0}
\]
yields an unbounded Kasparov $A\otimes\CCliff_1$-$B$-module $(\A_m\otimes\CCliff_1,(E\oplus E)_B,\bar{\til\D'})$ representing the same class as $(\A,E_B,\D)$, and furthermore $\bar{\til\D'}$ has compact resolvent.
\end{thm}
\begin{proof}
From \cref{lem:mult_vanish_inf} we know that $\A_m$ is dense in $A$, $[m,\phi_k] = 0$ (for all $k\in\N$), and $(\bar m\pm i)^{-1}$ lies in $A$. In particular, we can replace $\A$ by $\A_m$ without affecting the underlying Kasparov class. Define the closed operators
\begin{align*}
\til\D &:= \mattwo{0}{-i\D}{i\D}{0} , & \til M &:= \mattwo{0}{\bar M}{\bar M}{0} .
\end{align*}
The first statement then follows from \cref{thm:Kasp_stab_loc_bdd}. 
We need to check that $\bar{\til\D'}$ has compact resolvent. 
We will first show that $\Dom(\bar{\til\D+\til M}) \subset \Dom\til M$. 
By assumption, we have $\|[\D,\phi_k]\| < 4^{-k}$. 
Using the same argument as in the proof of \cite[Theorem 1.25]{MR16}, one shows that the commutator $[\D,\bar M]$ is well-defined and bounded on $\Ran(\D\pm i)^{-1}(\bar M\pm i)^{-1}$. 
It then follows from 
\cite[Theorem 6.1.8]{Mes14} (see also \cite[Proposition 7.7]{KL12}) 
that $(\pm i\D+\bar M)^* = \mp i\D+\bar M$, and in particular $\pm i\D+\bar M$ is closed on the domain $\Dom\D\cap\Dom\bar M$. Hence the operator $\til\D+\til M$ is closed on $\Dom\til\D\cap\Dom\til M$, which implies that $\Dom(\bar{\til\D+\til M})\subset\Dom\til M$, as desired. 

We know that $\til\D+\til M$ has locally compact resolvent, so in particular the operator $\phi_k(\til\D+\til M\pm i)^{-1}$ is compact. 
Consider the inequality
\[
\|(1-\phi_k)(\til\D+\til M\pm i)^{-1}\| 
\leq \|(1-\phi_k)(\bar m\pm i)^{-1}\| \, \| (\bar M\pm i)(\til M\pm i)^{-1} \| \, \|(\til M\pm i)(\til\D+\til M\pm i)^{-1}\| .
\]
Since $(\bar m\pm i)^{-1}$ lies in $A$ and $\phi_k$ is an approximate unit in $A$, the first factor on the right-hand-side converges to zero (as $k\to\infty$). 
By \cref{lem:dom_inc_rel_bdd}, the domain inclusion $\Dom(\bar{\til\D+\til M}) \subset \Dom\til M$ implies that $\til M$ is relatively bounded by $\til\D+\til M$, so the third factor is bounded. Similarly, the second factor is bounded because $\Dom\til M = \Dom\bar M\oplus\Dom\bar M$. 
It then follows that the resolvent $(\til\D+\til M\pm i)^{-1}$ is the norm limit of the compact operators $\phi_k(\til\D+\til M\pm i)^{-1}$, and therefore $\til\D+\til M$ has compact resolvent.
\end{proof}

\begin{remark}
With the assumptions of the above theorem, consider (the closure of) the operator $\D+M$. If we have the domain inclusion $\Dom\bar{\D+M} \subset \Dom\bar M$, then the same argument as in the above theorem shows that $(\A,E_B,\bar{\D+M})$ is an odd unbounded Kasparov $A$-$B$-module representing the same class as $(\A,E_B,\D)$, and that $\bar{\D+M}$ has compact resolvent. 
However, in general the domain inclusion $\Dom\bar{\D+M} \subset \Dom\bar M$ might not hold. 
\end{remark}

We prove a similar result for the case of \emph{even} unbounded Kasparov modules. Again, we need to `double up' the module (although this is somewhat less natural in the even case) to obtain the aforementioned domain inclusion. 
So, let $A$ and $B$ now be $\Z_2$-graded $C^*$-algebras, and consider an (even) unbounded Kasparov $A$-$B$-module $(\A,E_B,\D)$. Consider the unbounded Kasparov module $(M_2(\C),\C^2,0)$, where the $\Z_2$-grading on $\C^2=\C\oplus\C$ is such that the first summand is even and the second summand is odd. 
Then the (external) Kasparov product with $\alpha_2 := [(M_2(\C),\C^2,0)] \in\KK(M_2(\C),\C)$ implements the isomorphism \cite[\S5, Theorem 1]{Kas80b}
\[
\hot\alpha_2 \colon \KK(A,B) \xrightarrow{\simeq} \KK^2(A,B) \simeq \KK(M_2(\C)\hot A,B) ,
\]
where we have identified $\CCliff_2 = M_2(\C)$. In other words, the unbounded Kasparov product of $(\A,E_B,\D)$ with $(M_2(\C),\C^2,0)$, given by $(\A\hot M_2(\C),E_B\hot\C^2,\D\hot1)$ (where $\hot$ denotes the $\Z_2$-graded tensor product), represents the same class as $(\A,E_B,\D)$. This procedure provides us with a similar doubling trick as in the odd case, and we can prove the following.

\begin{thm}
Let $(\A,E_B,\D)$ be an (even) unbounded Kasparov $A$-$B$-module. 
Suppose that the representation $A\to\End_B(E)$ is non-degenerate. Let $\{\phi_k\}_{k\in\N}\subset\A$ be as in the Standing Assumptions, and assume that each $\phi_k$ is even. 
Suppose that $\|[\D,\phi_k]\| < 4^{-k}$ for all $k$. 
Let $M$ be the unbounded operator on $E$ corresponding to the unbounded multiplier $m := \sum_{k\in\N} 2^k(\phi_{k+1}-\phi_k)$. 
Write $\A_m := \A\cap\Dom m$. 
Then the operator
\begin{align*}
\til\D' &:= \D\hot1 + \bar M\hot e , & e &= \mattwo{0}{1}{1}{0} ,
\end{align*}
yields an unbounded Kasparov $A\hot M_2(\C)$-$B$-module $(\A_m\hot M_2(\C),E_B\hot\C^2,\til\D')$ representing the same class as $(\A,E_B,\D)$, and furthermore $\til\D'$ has compact resolvent.
\end{thm}
\begin{proof}
The idea is similar to \cref{thm:odd_Kasp_cpt_res}; the main difference is the presence of $\Z_2$-gradings. The anti-commutator of $\til\D:=\D\hot1$ and $\til M:=\bar M\hot e$ is given by $\{ \D\hot1 , \bar M\hot e \} = [\D,\bar M]\hot e$. We know from the proof of \cref{thm:odd_Kasp_cpt_res} that $[\D,\bar M]$ is bounded. Hence $\{\til\D,\til M\}$ is bounded, which by \cite[Theorem 6.1.8]{Mes14} (see also \cite[Proposition 7.7]{KL12}) implies that $\til\D'=\til\D+\til M$ is closed on $\Dom\til\D\cap\Dom\til M$. In particular, we have the domain inclusion $\Dom(\bar{\til\D+\til M}) \subset \Dom\til M$. The remainder of the argument is as in the proof of \cref{thm:odd_Kasp_cpt_res}.
\end{proof}

\begin{remark}
As in \cref{sec:odd_KK}, we can obtain a converse to the `doubling up' procedure which replaces $(\A,E_B,\D)$ by $(\A\hot M_2(\C),E_B\hot\C^2,\D\hot1)$. More precisely, given any unbounded Kasparov $A\hot M_2(\C)$-$B$-module $(\A\hot M_2(\C),E_B\hot\C^2,\til\D)$ (where $A$ acts non-degenerately and $M_2(\C)$ acts via the standard representation) and an adequate approximate identity for $\til\D$, one can show that $\til\D$ is equal to the sum of a self-adjoint operator $\D\hot1$ and a locally bounded symmetric operator. We leave the details to the reader. 
\end{remark}

\newcommand{\MR}[1]{}


\providecommand{\noopsort}[1]{}\providecommand{\vannoopsort}[1]{}
\providecommand{\bysame}{\leavevmode\hbox to3em{\hrulefill}\thinspace}
\providecommand{\MR}{\relax\ifhmode\unskip\space\fi MR }
\providecommand{\MRhref}[2]{%
  \href{http://www.ams.org/mathscinet-getitem?mr=#1}{#2}
}
\providecommand{\href}[2]{#2}

\end{document}